\documentclass[12pt]{amsart}
\usepackage{amssymb,latexsym, amssymb, amsthm,amsmath}
\usepackage{enumerate}
\usepackage{amsfonts}
\usepackage{color}
\usepackage{comment}

\usepackage{hyperref}

\newtheorem{lemma}{Lemma}
\newtheorem{theorem}[lemma]{Theorem}

\newtheorem{remark}[lemma]{Remark}

\numberwithin{equation}{section} \numberwithin{lemma}{section}

\newtheorem{example}[lemma]{Example}

\makeatletter
\@namedef{subjclassname@2010}{%
  \textup{2010} Mathematics Subject Classification}
\makeatother

\numberwithin{equation}{section}

\newcommand{\N}{\mathbb{N}}

\newcommand{\Z}{\mathbb{Z}}
\newcommand{\Q}{\mathbb{Q}}

\newcommand{\vektor}[1]{\mathbf{#1}}

\newcommand{\cv}{\vektor{c}}
\newcommand{\ev}{\vektor{e}}

\newcommand{\nv}{\vektor{n}}

\newcommand{\kv}{\vektor{k}}
\newcommand{\mv}{\vektor{m}}
\newcommand{\xv}{\vektor{x}}
\newcommand{\yv}{\vektor{y}}

\newcommand{\hv}{\vektor{h}}

\newcommand{\zv}{\vektor{z}}
\newcommand{\uv}{\vektor{u}}
\newcommand{\vv}{\vektor{v}}
\newcommand{\wv}{\vektor{w}}

\newcommand{\nullv}{\boldsymbol{0}}

\newenvironment{blau}{\color{blue}}{}
\newenvironment{gruen}{\color{green}}{}
\newenvironment{rot}{\color{red}}{}
\specialcomment{forme}{\begin{blau}}{\end{blau}}
\excludecomment{forme}
\specialcomment{exam}{\begin{gruen}}{\end{gruen}}
\excludecomment{exam}
\specialcomment{expl}{\begin{rot}}{\end{rot}}
\excludecomment{expl}

\begin{document}

\title[Bilinear forms]{Upper bounds for integer solutions to a system of two bilinear forms}

\author[Eugen Keil]{Eugen Keil}

\address{Mathematical Institute \\
University of Oxford \\
Andrew Wiles Building \\
Radcliffe Observatory Quarter \\
Woodstock Road \\
Oxford \\
OX2 6GG \\
United Kingdom}

\email{Eugen.Keil@maths.ox.ac.uk}

\date{\today}

\subjclass[2010]{Primary 11D09; Secondary 11D72} 
%11D09 Quadratic and bilinear equations  
%11D72 Equations in many variables
\keywords{bilinear forms \and upper bound \and structure}

\maketitle

\begin{abstract}
We show that the number of integer solutions for a pair of bilinear equations 
in at least $2\times6$ variables has (up to logarithms) the expected upper bound unless there is
a structural reason why it is not the case.
\end{abstract}

\section{Introduction}

In the by now classical work, Birch \cite{Birch} provides a method to show that a system
of forms of degree $d$ has the expected number of solutions as long as the number of variables
is big enough compared to the dimension of the `singular locus'. If we have a system of $R$ forms
in $n$ variables and $V^*$ is the `singular locus', then the condition is
\begin{align*}
n-dim[V^*] > R(R+1)(d-1)2^{d-1}.
\end{align*}

In this paper we want to consider the case of two bilinear forms in $2s$ variables.
The condition in this case would be $2s-dim[V^*] > 12$. It can probably be improved
by recent work of Schindler \cite{Schindler} on Birch's theorem for bihomogeneous forms.

From a na\"ive point of view, it seems very strange that the method should give weaker results
if we find ourselves in more structured situations with a large singular locus,
like the case of two diagonal forms, where the singular locus is at least as big as $s$. 
On the other hand, the standard circle method approach is well suited to answer 
the diagonal problem in as few as $2 \times 4$ variables.

The question we are trying to answer is: Can we prove a result for `all' bilinear forms,
independent of the size of the singular locus?
We show that this is indeed the case if we restrict our attention to
upper bounds instead of asymptotic formulas.

\begin{theorem} \label{Main-thm}
Let $N \geq 2$ and $B_1, B_2 \in \Z^{s \times s}$ be two integer matrices with 
$B_1(\xv,\yv) =  \xv^T B_1 \yv$, $B_2(\xv,\yv) =  \xv^T B_2 \yv$ corresponding bilinear 
forms, then the number of solutions to the system $B_1(\xv,\yv)=B_2(\xv,\yv)=0$ with
$|x_i| \leq N, |y_i| \leq N$ is of order $O(N^{2s-4} (\log N)^{2})$ as long as we are not in
one of the two situations
\begin{enumerate}
\item rank$(\lambda B_1 + \mu B_2) \leq 5$ for all $\lambda,\mu \in \Z$ or
\item rank$(\lambda B_1 + \mu B_2) \leq 1$ for some $(\lambda,\mu) \neq (0,0)$.
\end{enumerate}
\end{theorem}

\begin{remark}
The number of solutions is bounded from below by $cN^{2s-4}$ for some $c>0$ by an averaging argument.
See also Lemma \ref{lem-low} below.
\end{remark}

The first exceptional case in Theorem \ref{Main-thm} does not give 
a sharp theoretical bound, but reflects a limitation in our methods. We would expect
that the same result holds with a three replacing the five (which would be best possible).

The second roughly corresponds to one of the equations being of the form 
$xy=0$, which cannot `save' two variables, 
as required in the theorem (see also Lemma \ref{lem-low}).
In this case our result is best possible.

Rank conditions as those that appear in our theorem are typical
in this line of work as can be seen in previous work of  
Schmidt \cite{Schmidt} and Dietmann \cite{Dietmann}, who deal with systems of general
quadratic forms.

Most of the following arguments will extend to general
systems of bilinear forms, but we feel that the methods and ideas are best presented
in the simplest case of two equations.\\

{\bf Acknowledgements:}\\
We would like to thank the mathematical institute at the University of Oxford for providing good working conditions. 
The author was supported by the EPSRC grant EP/J009458/1.

\section{Collecting the Tools}

Before we begin with stating the main lemmata of this work, we need
a few notational conventions.

As usual, we use O-notation and the Vinogradov notation $f \ll g$ to denote that
$|f| \leq C|g|$ for some $C > 0$.
In the same way, we say that the number of solutions $S(N)$ is \emph{essentially bounded}
by a quantity $T(N)$, if there is a $C>0$ such that $S(N) \leq T(CN)$ is a bound for all $N \in \N$.

Now we want to state the tools that we are going to use excessively throughout the paper. 
Most of them are simple results from linear algebra.

\begin{lemma}[Homogenisation] \label{lem-homog}
Let $A \in \Z^{s \times s}$, $\cv \in \Z^s$ and $y_i \in \Z, |y_i| \leq N$.
\begin{itemize}
\item[(i)] Let $A \yv=\cv$ be a system of inhomogeneous linear equations.
Then the number of solutions to this equation is essentially bounded by the number of
solutions to the homogeneous system $A \yv=\nullv$.
\item[(ii)] Let $A \yv=\nullv$ be a system of linear equations.
Then the number of solutions to this equation is essentially bounded by $N$ times
the number of solutions to the same system with $y_j=0$ for some $1\leq j \leq s$.
\item[(iii)] If the last $d$ entries of $A\yv$ don't depend on the variables
$y_1,\ldots,y_j$, then we can set the variables $y_{j+1},\ldots,y_s$ equal to zero
in the first $r-d$ equations of $A\yv = \nullv$ and obtain an 
essential upper bound for the number of solutions.
In other words: If $A$ is a upper triangular block matrix, we can change it into a diagonal
block matrix.
\end{itemize}
\end{lemma}

\begin{remark}
A bilinear system $\xv^T B_i \yv=0$ can always be thought of as a linear system
in $\yv$ by fixing the variables $\xv$ (or the other way around).
\end{remark}

\begin{proof}

For the first statement, we observe that for a given fixed solution $A \zv=\cv$
and any other solution $A \yv=\cv$ to the inhomogeneous linear equation, 
we obtain a solution $A (\yv-\zv)=\nullv$ with $\|\yv-\zv\|_{\infty} \leq 2N$.

For the second statement, we observe that by fixing $y_j$, we can rewrite 
$A\yv=\nullv$ into $B\yv'=\cv_j$, where $B$ is essentially $A$ but with missing column $j$ and
$\cv_j$ is $-y_j$ times the $j$th column of $A$. The result follows from part one and the
observation, that there are $O(N)$ choices for $y_j$.

The third statement is slightly more difficult. For any choice of values for
$\yv''=(y_{j+1},\ldots,y_s)$ that satisfy the last $d$ equations of $A\yv=\nullv$,
we can set $\yv'=(y_1,\ldots,y_j)$ and write the first $r-d$ equations in the form
$B \yv' = c(\yv'')$, where $B$ is the upper left submatrix of size $(r-d) \times j$. 
By part (i), this system is majorized by the system $B \yv' = 0$.
This homogenisation procedure doesn't affect the last $d$
equations since they are independent of $\yv'$.
\end{proof}

\begin{lemma}[Divisor estimates] \label{lem-divest}
An equation of the form $dx_1y_1=cx_2y_2$ with $c,d \neq 0$ has $O(N^2\log N)$
solutions with $|x_i|,|y_i| \leq N$.
\end{lemma}

\begin{proof}
A Cauchy-Schwarz symmetrisation (see proof of Lemma \ref{lem-diag} below) 
reduces the problem to the form $x_1y_1=x_2y_2$. 
For $x_1 = 0$, the number of solutions is $O(N^2)$. We can therefore assume
that $x_1$ and $x_2$ are non-zero and positive. If we consider
this to be a linear equation in $y_1$ and $y_2$ and set $d = \gcd(x_1,x_2)$, 
then we can instead look at $u_1y_1=u_2y_2$, where $\gcd(u_1,u_2)=1$ and $u_i = x_i/d$.
This forces the divisibility conditions $u_1|y_2$ and $u_2|y_1$. Therefore, the
number of solutions to this linear diophantine equation is bounded by $(2N+1)/\max(u_1,u_2)$.
We obtain an essential upper bound of the form
\begin{align*}
N \sum_{1 \leq x_1,x_2 \leq N} \frac{\gcd(x_1,x_2)}{\max(x_1,x_2)}.
\end{align*}
Collecting the terms with equal greatest common divisor, we obtain
\begin{align*}
& N \sum_{d \leq N} \sum_{1 \leq u_1,u_2 \leq N/d} \frac{1}{\max(u_1,u_2)}
\leq 2 N \sum_{d \leq N} \sum_{1 \leq u_1 \leq u_2 \leq N/d} \frac{1}{u_2}\\
= & 2 N \sum_{d \leq N} \sum_{1 \leq u_2 \leq N/d} 1 
= 2 N \sum_{d \leq N} \frac{N}{d} \ll N^2 \log N.
\end{align*}
\end{proof}

The next lemma isn't strictly necessary for the argument, but 
simplifies the exposition.

\begin{lemma}[Coordinate change] \label{lem-coord}
If we set $\uv = K \xv$ for a matrix $K \in \Z^{s\times s}$ of full rank, 
then the resulting system $\uv K^{-1}B\yv = \nullv$ has a bigger upper bound, 
as long as we choose $C>0$ with
$|u_i| \leq CN$ in such a way that $[-CN,CN]^s$ covers the image of $[-N,N]$ by $K$. 
By multiplication with suitable integers, we can
also assume that the coefficients of the new system are integers. 
\end{lemma}

\begin{proof}
Every solution in $\xv$ translates into a solution in $\uv$.
\end{proof}

\begin{lemma}[Lower bound] \label{lem-low}
A system $\xv^TB_j\yv=0$ with $1 \leq j \leq r$ in $2s$ variables has 
$\gg N^{2(s-r)}$ many solutions with $|x_i|,|y_i| \leq N$.
\end{lemma}

\begin{proof}
Lemma \ref{lem-homog} shows us that the number of solutions to the system
$\xv^TB_j\yv=h_j$ for arbitrary fixed $h_j$ is essentially bounded by 
the number of solutions to the system $\xv^TB_j\yv=0$.
If we now consider $h_j$ to be variables as well, which have the range
$|h_j| \leq C_jN^2$ for some large enough $C_j > 0$ (depending on $B_j$), then the 
total number of solutions to the system $\xv^TB_j\yv=h_j$ is $N^{2s}$ since
we can choose $x_i$ and $y_i$ freely and this choice fixes 
the values of all $h_j$. We obtain
\begin{align*}
N^{2s} = & \sum_{\hv} \# \{x_i,y_i: \xv^TB_j\yv=h_j\}
\ll \sum_{\hv} \# \{x_i,y_i: \xv^TB_j\yv=0\} \\
\ll & N^{2r} \# \{x_i,y_i: \xv^TB_j\yv=0\}.
\end{align*}
\end{proof}

\begin{lemma}[Diagonal system] \label{lem-diag}
The system
\begin{align*}
d_1x_1y_1 + d_2x_2y_2 = d_3x_3y_3 + d_4x_4y_4,\\
e_1x_1y_1 + e_2x_2y_2 = e_3x_3y_3 + e_4x_4y_4,
\end{align*}
has $O(N^4 (\log N)^2)$ solutions with $|x_i|,|y_i| \leq N$ if and only if 
every $2\times3$ submatrix of 
\begin{align*}
\begin{pmatrix}
d_1 & d_2 & d_3 & d_4 \\ 
e_1 & e_2 & e_3 & e_4
\end{pmatrix}
\end{align*}
has rank two.
\end{lemma}

\begin{proof}
Let us first assume that every $2\times 3$ submatrix has rank two.
This implies that we can rearrange matters such that
the submatrices  
$\begin{pmatrix}
d_1 & d_2 \\ 
e_1 & e_2
\end{pmatrix}$
and
$\begin{pmatrix}
d_3 & d_4 \\ 
e_3 & e_4
\end{pmatrix}$
have rank two.
Write the number of solutions to the system as a sum and perform a simple 
Cauchy-Schwarz symmetrisation of the coefficients. 
\begin{align*}
& \sum_{|\xv| \leq N} \sum_{|\yv| \leq N} \sum_{\substack{n,m \\ d_1x_1y_1 + d_2x_2y_2 = n=d_3x_3y_3 + d_4x_4y_4 \\ e_1x_1y_1 + e_2x_2y_2 =m= e_3x_3y_3 + e_4x_4y_4}} 1 \\
= & \sum_{n,m} \Bigg(\sum_{\substack{d_1x_1y_1 + d_2x_2y_2 = n \\ e_1x_1y_1 + e_2x_2y_2 = m}} 1 \Bigg) 
\cdot
\Bigg(\sum_{\substack{d_3x_3y_3 + d_4x_4y_4=n \\ e_3x_3y_3 + e_4x_4y_4=m}} 1 \Bigg)\\
\leq & \left(\sum_{n,m} \Bigg(\sum_{\substack{d_1x_1y_1 + d_2x_2y_2 = n \\ e_1x_1y_1 + e_2x_2y_2 = m}}  1 \Bigg)^2 \right)^{1/2} \cdot 
\left(\sum_{n,m} \Bigg(\sum_{\substack{d_3x_3y_3 + d_4x_4y_4=n \\ e_3x_3y_3 + e_4x_4y_4=m}} 1 \Bigg)^2 \right)^{1/2}
\end{align*}
The first term corresponds to the system
\begin{align*}
d_1x_1y_1 + d_2x_2y_2 = d_1x_3y_3 + d_2x_4y_4,\\
e_1x_1y_1 + e_2x_2y_2 = e_1x_3y_3 + e_2x_4y_4,
\end{align*}
and by taking linear combinations, we can reduce this to
\begin{align*}
x_1y_1 = x_3y_3,\\
x_2y_2 = x_4y_4,
\end{align*}
which has the required number of solutions by Lemma \ref{lem-divest}.
An equivalent argument takes care of the other term.

In the case that there is a $2 \times 3$ submatrix with rank at most one, we
can simplify the system to 
\begin{align*}
d_1x_1y_1 + d_2x_2y_2 &= d_3x_3y_3 + d_4x_4y_4,\\
0 &= e_4x_4y_4.
\end{align*}
If $e_4=0$ we can show by the argument from Lemma \ref{lem-low}
that the first equation has $\gg N^{6}$ solutions. In the case $e_4 \neq 0$
we have $x_4=0$ (or $y_4=0$), which reduces the first equation to six variables
and a factor of $\gg N$ from the summation over $y_4$ (or $x_4$). Again by Lemma
\ref{lem-low} we have at least $\gg N^{5}$ solutions.  
\end{proof}

The last ingredient is another simple observation about systems of linear equations.

\begin{lemma} \label{lem-sum}
Let $A \yv = \cv$ be a linear equations system, where all equations
are independent of $y_s$ apart from the last equation. In other words,
the last column of $A$ is a non-zero multiple of the standard basis vector 
$\ev_s$. Then the number of solutions to this system is bounded by
the number of solutions to $A'\yv'=\cv'$, where we get $A'$ by removing the last column
of $A$ and $\yv',\cv'$ by removing the last entry of $\yv$ and $\cv$.  
\end{lemma}

\begin{proof}
For fixed values of the variables $y_1,\ldots,y_{s-1}$, there is at most one
value of $y_s$ that makes the last equation true.
\end{proof}

\section{The starting point} \label{Sec-start}

It turns out that it is sufficient to prove the main result for $s = 6$
in order to get it for all $s \geq 6$ as we will see in Section \ref{Sec-deduction}.
We therefore begin with the system of two bilinear equations in two times six variables 
\begin{align*}
B_1(\xv,\yv)=B_2(\xv,\yv)=0.
\end{align*}
By taking linear combinations of the two equations, we can assume that rank$(B_1) = 6$
as long as we are not in the first exceptional case of Theorem \ref{Main-thm}.

One way to look at the system is to consider them as linear equations in $\yv$ with
coefficients being linear forms in $\xv$. We get
\begin{align*}
K_1(\xv)y_1 + K_2(\xv)y_2 + K_3(\xv)y_3 + K_4(\xv)y_4 + K_5(\xv)y_5 + K_6(\xv)y_6 = 0,\\
L_1(\xv)y_1 + L_2(\xv)y_2 + L_3(\xv)y_3 + L_4(\xv)y_4 + L_5(\xv)y_5 + L_6(\xv)y_6 = 0.
\end{align*}
By a change of coordinates (Lemma \ref{lem-coord}), 
we can assume that the first equation is diagonal. This simplifies the situation to
\begin{align*}
x_1y_1 + x_2y_2 + x_3y_3 + x_4y_4 + x_5y_5 + x_6y_6 = 0,\\
L_1(\xv)y_1 + L_2(\xv)y_2 + \ldots + L_5(\xv)y_5 + L_6(\xv)y_6 = 0.
\end{align*}

The main difficulty to overcome is the interdependence of the two equations. 
Our goal will be to either extract a system with separated variables 
or one with a diagonal structure.

One possible way to force independence is to set set $x_1=x_2=0$.
With $\xv' = (0,0,x_3,x_4,x_5,x_6)$
Lemma \ref{lem-homog} (ii) gives us a factor of $O(N^2)$ and the system
\begin{align*}
x_3y_3 + x_4y_4 + x_5y_5 + x_6y_6 = 0,\\
L_1(\xv')y_1 + L_2(\xv')y_2 + L_3(\xv')y_3 + \ldots + L_6(\xv')y_6 = 0.
\end{align*}
This makes the first equation independent of $y_1$ and $y_2$.
By Lemma \ref{lem-homog} (iii), we can remove the dependence of the second
equation on $y_3,\ldots,y_6$, which leaves us with
\begin{align*}
x_3y_3 + x_4y_4 + x_5y_5 + x_6y_6 = 0,\\
L_1(\xv')y_1 + L_2(\xv')y_2  = 0.
\end{align*}
Now that the second equation is independent of $y_5$ and $y_6$, we can do the same
thing to the first line and obtain the majorising system
\begin{align*}
x_5y_5 + x_6y_6 = 0,\\
L_1(\xv')y_1 + L_2(\xv')y_2  = 0.
\end{align*}
If $L_1$ and $L_2$ depend on $(x_3,x_4)$ in
a non-singular way (see below), a final change of variables would give us the system
\begin{align*}
x_5y_5 + x_6y_6 = 0,\\
x_3y_1 + x_4y_2  = 0.
\end{align*}
Since we achieved independence, Lemma \ref{lem-divest} gives us $O(N^4(\log N)^2)$ solutions.
To obtain the final bound, we collect the $O(N^2)$ contribution from our first step
and the $O(N^2)$ from the sum over $(y_3,y_4)$.\\

Let us explore the conditions under which the above argument works.
Write $l_{ij}$ for the $j$th coefficient of $L_i$. 
So that $L_1(\xv) = l_{11}x_1 + \ldots + l_{16}x_6$ and $L_2(\xv) = l_{21}x_1 + \ldots + l_{26}x_6$.
If the matrix$\begin{pmatrix}l_{13} & l_{14} \\ l_{23} & l_{24}\end{pmatrix}$ has rank two,
then the change of variables from $(L_1(\xv'),L_2(\xv'))$ to $(x_3,x_4)$ will be successful.

To understand the complementary case, we observe that we made some arbitrary choices along the way.
Consider the `off-diagonal' matrix build from 
the coefficients of the linear forms $L_1$ and $L_2$, given by
\begin{align} \label{eq-offdiag-matrix}
\begin{pmatrix}
l_{13} & l_{14} & l_{15} & l_{16} \\ 
l_{23} & l_{24} & l_{25} & l_{26}
\end{pmatrix}.
\end{align}
Whenever this matrix has full rank, the above strategy will also work by choosing a
a possibly different pair of indices than $3$ and $4$, which corresponds to the 
special case that the first two columns are linearly independent.
So in order for this not to work, we need that matrix \eqref{eq-offdiag-matrix}
has rank at most one.

On the other hand, we can set any pair of variables $\{x_i,x_j\}$
equal to zero in the first step of the argument, not necessarily $x_1$ and $x_2$.
Since the matrix \eqref{eq-offdiag-matrix} sits in the upper right corner of $B_2$, 
this translates (by permuting the variables) into the following 
rank condition for the matrix $B_2$:
Any off-diagonal matrix in $B_2$ has rank at most one.
('Off-diagonal' means that it doesn't contain any diagonal elements.)

Write $ \vv \otimes \wv := \vv \cdot \wv^T$, $\ev_i$ to be the $i$th standard basis vector
and define the \emph{off-rank} of a matrix to be the maximal rank of an off-diagonal submatrix.
We have the following classification of off-rank one matrices.

\begin{lemma} \label{Rank1-structure}
A matrix $B \in \Z^{s \times s}$ with off-rank one has the form
\begin{enumerate}
\item[(i)] $B = D + \vv \otimes \wv$,
\item[(ii)] $B = D + \vv \otimes \ev_i + \ev_i \otimes \wv$ or
\item[(iii)] $B = D + E$,
\end{enumerate}
where $D$ is a diagonal matrix, $\vv,\wv \in \Q^s$ and $E \in \Z^{s \times s}$ has
non-zero entries only in a $3 \times 3$ submatrix, which is based on the diagonal.
\end{lemma}

\begin{proof}
See Appendix \ref{AppA}.
\end{proof}

\begin{example}
To get a better feeling for this concept, we give examples of the three possible cases.
\begin{align*}
\begin{pmatrix}
0 & 2 & 2 & 2 & 1\\ 
3 & -1 & 6 & 6 & 3\\
0 & 0 & 6 & 0 & 0 \\ 
1 & 2 & 2 & -5 & 1\\
1 & 2 & 2 & 2 & 3
\end{pmatrix},
\begin{pmatrix}
3 & 10 & 0 & 0 & 0\\ 
3 & 7 & 5 & 2 & 8\\
0 & -2 & 5 & 0 & 0\\ 
0 & -3 & 0 & 1 & 0\\
0 & 2 & 0 & 0 & -4
\end{pmatrix},  
\begin{pmatrix}
1 & 17 & 2 & 0 & 0\\ 
3 & -1 & 1 & 0 & 0\\
4 & -1 & 6 & 0 & 0\\ 
0 & 0 & 0 & 6 & 0\\
0 & 0 & 0 & 0 & 2
\end{pmatrix}.
\end{align*}
We have $\vv = (1,3,0,1,1)^T$ and $\wv=(1,2,2,2,1)^T$ for the first example.
\end{example}

Each of the next three sections is dealing with one of the cases in \ref{Rank1-structure}. 

\begin{remark}
The off-rank zero case for $B_2$ is covered by any of the following sections.
\end{remark}

\section{Diagonal case (i)}

In order to understand the structure of $B_2 = D + \vv \otimes \wv$ we introduce
new variables $h = \xv^T \vv$ and $l = \wv^T \yv$. Then the bilinear system
transforms into
\begin{align*}
x_1y_1 + x_2y_2 + x_3y_3 + x_4y_4 + x_5y_5 + x_6y_6 &= 0,\\
d_1x_1y_1 + d_2x_2y_2 + d_3x_3y_3 + d_4x_4y_4 + d_5x_5y_5 + d_6x_6y_6 &= hl,\\
v_1x_1 + v_2x_2 + v_3x_3 + v_4x_4 + v_5x_5 + v_6x_6 &= h,\\
w_1y_1 + w_2y_2 + w_3y_3 + w_4y_4 + w_5y_5 + w_6y_6 &= l.
\end{align*}
This system has now the advantage of being diagonal, while having
a higher complexity due to the two additional linear equations. 

The exact behaviour of this system depends on the coefficients $d_i, v_i$ and $w_i$.
We use Lemma \ref{lem-homog} to set $x_1=y_2=0$ similar to the procedure 
in the non-degenerate case in Section \ref{Sec-start}.
The linear equations in the system

\begin{align*}
 x_3y_3 + x_4y_4 + x_5y_5 + x_6y_6 &= 0,\\
 d_3x_3y_3 + d_4x_4y_4 + d_5x_5y_5 + d_6x_6y_6 &= hl,\\
v_2x_2 + v_3x_3 + v_4x_4 + v_5x_5 + v_6x_6 &= h,\\
w_1y_1 + w_3y_3 + w_4y_4 + w_5y_5 + w_6y_6 &= l,
\end{align*}
can be dealt with by Lemma \ref{lem-sum} as long as
the coefficients $v_2$ and $w_1$ are non zero.
We end up with the reduced problem of bounding the solutions to
\begin{align*}
 x_3y_3 + x_4y_4 + x_5y_5 + x_6y_6 &= 0,\\
 d_3x_3y_3 + d_4x_4y_4 + d_5x_5y_5 + d_6x_6y_6 &= hl.
\end{align*}
By Lemma \ref{lem-diag} (and Lemma \ref{lem-homog}) we have $O(N^6(\log N)^2)$ solutions as long as
not all $d_i$ are equal. Together with the $O(N^2)$ contribution from Lemma \ref{lem-homog}
in the first step, we obtain the result.\\

As in the previous section, we need to analyse the argument to obtain a good description
of the complementary case.
The method words if $v_2,w_1 \neq 0$ and $d_i \neq d_j$ for some $i,j \in \{3,4,5,6\}$.
By symmetry (remaining of variables), we can perform the argument with different sets of indices as well.

The first step succeeds, therefore, if there are $v_i$ and $w_j$ with $i \neq j$,
which are both non-zero. Let us explore the complementary situation.\\

{\bf Case 1}: $v_i\cdot w_j = 0$ for all $i \neq j$.\\
This implies that either $\vv = \nullv$, $\wv = \nullv$ or that
$\vv$ and $\wv$ have only one non-zero component with the same index.\\

{\bf Case 1.1}: $\wv = \nullv$ ($\vv = \nullv$ is equivalent by symmetry).\\
Since $B_2 = D + \vv \otimes \wv$, we get the system
\begin{align*}
x_1y_1 + x_2y_2 + x_3y_3 + x_4y_4 + x_5y_5 + x_6y_6 &= 0,\\
d_1x_1y_1 + d_2x_2y_2 + d_3x_3y_3 + d_4x_4y_4 + d_5x_5y_5 + d_6x_6y_6 &= 0.
\end{align*}
Lemma \ref{lem-diag} and Lemma \ref{lem-homog} give the right answer
as long as the $d_i$ take on three different values.
If there are only two different values for $d_i$, at least three of the
coefficients have to be the same and a linear combination (with a renaming of variables) 
brings us to
\begin{align*}
x_1y_1 + x_2y_2 + x_3y_3 + x_4y_4 + x_5y_5 + x_6y_6 &= 0,\\
d_4x_4y_4 + d_5x_5y_5 + d_6x_6y_6 &= 0,
\end{align*}
where $d_4,d_5 \in \{0,d_6\}$. By Lemma \ref{lem-homog} (iii) this simplifies
further to
\begin{align*}
x_1y_1 + x_2y_2 = 0 = d_4x_4y_4 + d_5x_5y_5 + d_6x_6y_6.
\end{align*}
If $d_5=d_6 \neq 0$ or  $d_4=d_6 \neq 0$, then we get the correct upper bound
by Lemma \ref{lem-divest}.
Otherwise, we end up with at most one non-zero coefficient, which brings us
into the second exceptional case of Theorem \ref{Main-thm}.\\

{\bf Case 1.2}: $w_j=v_i=0$ for $i,j \geq 2$ (similar cases by coordinate change).\\
The linear equations simplify to $v_1x_1 = h$ and $w_1y_1 = l$
and the whole system changes into
\begin{align*}
x_1y_1 + x_2y_2 + x_3y_3 + x_4y_4 + x_5y_5 + x_6y_6 &= 0,\\
(d_1-v_1w_1)x_1y_1 + d_2x_2y_2 + d_3x_3y_3 + d_4x_4y_4 + d_5x_5y_5 + d_6x_6y_6 &= 0.
\end{align*}
This is the same situation we faced in the previous case and can be dealt with accordingly.\\

Now we are going to discuss the second part of the general argument in this section, 
where we needed that at least one of the coefficients $d_i$ is non-zero for $3 \leq i \leq 6$.
What happens if this is not the case?\\

{\bf Case 2}: $v_1,w_2 \neq 0$, but $d_3=d_4=d_5=d_6=0$.\\
The system simplifies to
\begin{align*}
x_1y_1 + x_2y_2 + x_3y_3 + x_4y_4 + x_5y_5 + x_6y_6 &= 0,\\
d_1x_1y_1 + d_2x_2y_2 &= hl,\\
v_1x_1 + v_2x_2 + v_3x_3 + v_4x_4 + v_5x_5 + v_6x_6 &= h,\\
w_1y_1 + w_2y_2 + w_3y_3 + w_4y_4 + w_5y_5 + w_6y_6 &= l.
\end{align*}
If any one of the coefficients $v_3,\ldots,v_6$ is non-zero, we can perform
the same argument to conclude that $d_1=0$.
A non-zero coefficient among $w_3,\ldots,w_6$ implies $d_2 = 0$.
This would imply that $hl=0$ and lead to the system
\begin{align*}
x_1y_1 + x_2y_2 + x_3y_3 + x_4y_4 + x_5y_5 + x_6y_6 &= 0,\\
(\vv^T \cdot \xv)(\wv^T \cdot \yv) &= 0.
\end{align*}
It corresponds to the degenerate case $(2)$ in Theorem \ref{Main-thm}.
Therefore, we may assume (for example) that $w_3,\ldots,w_6$ are all zero.
\begin{align*}
x_1y_1 + x_2y_2 + x_3y_3 + x_4y_4 + x_5y_5 + x_6y_6 &= 0,\\
d_1x_1y_1 + d_2x_2y_2 &= hl,\\
v_1x_1 + v_2x_2 + v_3x_3 + v_4x_4 + v_5x_5 + v_6x_6 &= h,\\
w_1y_1 + w_2y_2 &= l.
\end{align*}
Assume for now that $v_3 \neq 0$. If we set $x_2=y_3=0$ by 
using Lemma \ref{lem-homog}, we end up with
\begin{align*}
x_1y_1 + x_4y_4 + x_5y_5 + x_6y_6 &= 0,\\
d_1x_1y_1 &= hl,\\
v_1x_1 + v_3x_3 + v_4x_4 + v_5x_5 + v_6x_6 &= h,\\
w_1y_1 + w_2y_2 &= l.
\end{align*}
The variables $x_3$ and $y_2$ have non-zero coefficients and appear only in linear equations. 
This allows us to use Lemma \ref{lem-sum} to reduce the system to
\begin{align*}
x_1y_1 + x_4y_4 + x_5y_5 + x_6y_6 &= 0,\\
d_1x_1y_1 &= hl,
\end{align*}
where the number of solutions is bounded by $O(N^6(\log N)^2)$ by Lemma \ref{lem-diag}
as long as $d_1 \neq 0$.

The same argument works if one of $v_4,\ldots,v_6$ is non-zero. Therefore,
we are doing fine, except when $d_1=0$ or $v_3=v_4=v_5=v_6=0$.\\

{\bf Case 2.1}: $v_3=v_4=v_5=v_6=0$.\\
By replacing the auxiliary variables $h$ and $l$, the system is now given by
\begin{align*}
x_1y_1 + x_2y_2 + x_3y_3 + x_4y_4 + x_5y_5 + x_6y_6 &= 0,\\
d_1x_1y_1 + d_2x_2y_2 &= (v_1x_1 + v_2x_2)(w_1y_1 + w_2y_2).
\end{align*}
Since the second equation is independent of the variables $y_5$ and $y_6$,
we can use  Lemma \ref{lem-homog} (iii) and Lemma \ref{lem-divest}
to bound the contribution of the first equation by $O(N^6\log N)$
independent of the variables $x_1,y_1,x_2,y_2$ and consider the equation
\begin{align*}
d_1x_1y_1 + d_2x_2y_2 &= (v_1x_1 + v_2x_2)(w_1y_1 + w_2y_2)
\end{align*}
on its own.
If the rank of the corresponding matrix is two, then Lemma \ref{lem-coord}
and Lemma \ref{lem-divest} will give the correct upper bound.
If, on the other hand, the rank is one, the we are again in the 
exceptional case $(2)$ in Theorem \ref{Main-thm}.\\

{\bf Case 2.2}: $d_1=0$ and $v_3 \neq 0$ (for example).\\
We took another small step forward in removing one more coefficient
from the second bilinear equation. The system now looks like
\begin{align*}
x_1y_1 + x_2y_2 + x_3y_3 + x_4y_4 + x_5y_5 + x_6y_6 &= 0,\\
d_2x_2y_2 &= hl,\\
v_1x_1 + v_2x_2 + v_3x_3 + v_4x_4 + v_5x_5 + v_6x_6 &= h,\\
w_1y_1 + w_2y_2 &= l.
\end{align*}
We can also assume that $d_2 \neq 0$ since the complementary case is
covered earlier in `Case 2'.
The final case analysis is whether $w_1 = 0$ or not.
If $w_1=0$, we obtain
\begin{align*}
x_1y_1 + x_2y_2 + x_3y_3 + x_4y_4 + x_5y_5 + x_6y_6 &= 0,\\
(d_2x_2-w_2h)y_2 &= 0,\\
v_1x_1 + v_2x_2 + v_3x_3 + v_4x_4 + v_5x_5 + v_6x_6 &= h.
\end{align*}
If we insert the linear equation into the second equation, we see that the corresponding
matrix has rank one. Therefore, we are in the exceptional case $(2)$ of Theorem \ref{Main-thm}.

If $w_1 \neq 0$, on the other hand, we can set $x_1 = 0=y_3$ by Lemma \ref{lem-homog}
and reduce the problem to
\begin{align*}
x_2y_2 + x_4y_4 + x_5y_5 + x_6y_6 &= 0,\\
d_2x_2y_2 &= hl,\\
v_2x_2 + v_3x_3 + v_4x_4 + v_5x_5 + v_6x_6 &= h,\\
w_1y_1 + w_2y_2 &= l.
\end{align*}
The variables $x_3$ and $y_1$ have non-zero coefficients and Lemma \ref{lem-sum}
allows us to remove the linear equations.
The remaining system
\begin{align*}
x_2y_2 + x_4y_4 + x_5y_5 + x_6y_6 &= 0,\\
d_2x_2y_2 &= hl,
\end{align*}
has $O(N^{6}(\log N)^2)$ solutions by Lemma \ref{lem-diag}.

\section{Parameter Case (ii)}

Now we have $B = D + \vv \otimes \ev_i + \ev_i \otimes \wv$. By a change of variables,
we can assume that $i=1$ and obtain the form
\begin{align*}
x_1y_1 + x_2y_2 + x_3y_3 + x_4y_4 + x_5y_5 + x_6y_6 &= 0,\\
d_1x_1y_1 + d_2x_2y_2 + d_3x_3y_3 + d_4x_4y_4 + d_5x_5y_5 + d_6x_6y_6 &= y_1L(\xv)+x_1M(\yv)
\end{align*}
for the linear forms $L(\xv)=\vv^T\xv$ and $M(\yv)=\wv^T\yv$.
The approach here is similar to the one in the previous section.

First we set $x_1=y_1=0$ with the help of Lemma \ref{lem-homog} and analyse the
simpler system
\begin{align*}
x_2y_2 + x_3y_3 + x_4y_4 + x_5y_5 + x_6y_6 &= 0,\\
d_2x_2y_2 + d_3x_3y_3 + d_4x_4y_4 + d_5x_5y_5 + d_6x_6y_6 &= 0.
\end{align*}
If the $d_i$ take on more than two values, we are done by Lemma \ref{lem-diag}.
Otherwise, we can take linear combinations and simplify further to
\begin{align*}
x_2y_2 + x_3y_3 + x_4y_4 + x_5y_5 + x_6y_6 &= 0,\\
d_5x_5y_5 + d_6x_6y_6 &= 0,
\end{align*}
where $d_5 \in \{0,d_6\}$ (after a renaming of variables).

If $d_5 = d_6 \neq 0$, we are given the right upper bound by Lemma \ref{lem-diag} again.
Otherwise, we have $d_5 = 0$ and have found that our original system must have the form
\begin{align*}
x_1y_1 + x_2y_2 + x_3y_3 + x_4y_4 + x_5y_5 + x_6y_6 &= 0,\\
d_6x_6y_6 &= y_1L(\xv)+x_1M(\yv)
\end{align*}

Here, we can use Lemma \ref{lem-homog} to set $x_1=0$. This makes the second equation 
independent of $y_2,\ldots,y_5$. By Lemma \ref{lem-homog} (iii), this implies
that we can simplify the system to
\begin{align*}
x_2y_2+x_3y_3 + x_4y_4 + x_5y_5 &= 0,\\
d_6x_6y_6 &= y_1 L(0,x_2,\ldots,x_6).
\end{align*}
If $v_i \neq 0$ for some $i \in \{2,3,4,5\}$, we can apply Lemma \ref{lem-homog} (iii)
again to remove the term $x_iy_i$ from the first equation and then 
change coordinates with Lemma \ref{lem-coord} to obtain (here $i=2$ for example)
\begin{align*}
x_3y_3 + x_4y_4 + x_5y_5 &= 0,\\
d_6x_6y_6 &= y_1 x_2,
\end{align*}
which has the right upper bound for the number of solutions by Lemma \ref{lem-diag}
as long as $d_6 \neq 0$.

This implies that we have the correct upper bound, except if $d_6=0$ or 
$L(\xv)=L(x_1,0,0,0,0,x_6)$.
A symmetric argument gives us the same conclusion with the condition 
$d_6=0$ or $M(\yv)=M(y_1,y_6)$.\\

{\bf Case 1}: $d_6 = 0$.\\
The system now simplifies to 
\begin{align*}
x_1y_1 + x_2y_2 + x_3y_3 + x_4y_4 + x_5y_5 + x_6y_6 &= 0,\\
y_1L(\xv)+x_1M(\yv) &= 0.
\end{align*}
Here we need a slightly unusual procedure. We set $h=L(\xv)$ and $l=M(\yv)$
to lift the system to
\begin{align*}
x_1y_1 + x_2y_2 + x_3y_3 + x_4y_4 + x_5y_5 + x_6y_6 &= 0,\\
y_1h+x_1l &= 0,\\
L(\xv) &= h,\\
M(\yv) &= l.
\end{align*}
Now we perform Lemma \ref{lem-homog} (iii) two times. One time with the set
$\{x_2,\ldots,x_6\}$ and a second time with $\{y_2,\ldots,y_6\}$.
This leaves us with the homogeneous system
\begin{align*}
x_2y_2 + x_3y_3 + x_4y_4 + x_5y_5 + x_6y_6 &= 0,\\
y_1h+x_1l &= 0,\\
L(0,x_2,\ldots,x_6) &= 0,\\
M(0,y_2,\ldots,y_6) &= 0.
\end{align*}
The second equation has four independent variables, which gives $O(N^2\log N)$
by Lemma \ref{lem-divest}.
The remaining system is an intersection of a diagonal bilinear form in $2\cdot5$
variables with two linear equations. The resulting bilinear form has rank at least three
and if those two equations aren't degenerate, we have the correct upper bound by 
Lemma \ref{lem-divest}.

Degenerate means here that one of the original linear forms $L$ of $M$ has to depend
only on $x_1$ or $y_1$. This would lead to a system of the shape
\begin{align*}
x_1y_1 + x_2y_2 + x_3y_3 + x_4y_4 + x_5y_5 + x_6y_6 &= 0,\\
(v_1y_1+M(\yv))x_1 &= 0,
\end{align*}
(or the equivalent for $M(\yv)=w_1y_1$), which has rank one in the second equation
and corresponds to the exceptional case $(2)$ in Theorem \ref{Main-thm}.\\

{\bf Case 2}: $v_i=w_i=0$ for $i \in \{2,3,4,5\}$.\\
In this case, we are left with the system
\begin{align*}
x_1y_1 + x_2y_2 + x_3y_3 + x_4y_4 + x_5y_5 + x_6y_6 &= 0,\\
(v_1+w_1)x_1y_1+d_6x_6y_6 &= v_6x_6y_1+w_6x_1y_6.
\end{align*}

We have seen this before in Case 2.1 of the previous section. Lemma \ref{lem-homog}
with Lemma \ref{lem-diag} are sufficient to deal with it.

\section{Pertubation Case (iii)}

In this last case we have $B = D + E$, where (we can assume that)
$E$ has only non-zero entries in the upper left $3\times3$ corner.
This corresponds to a system of the form
\begin{align*}
x_1y_1 + x_2y_2 + x_3y_3 + x_4y_4 + x_5y_5 + x_6y_6 &= 0,\\
L_1(\xv')y_1 + L_2(\xv')y_2 + L_3(\xv')y_3 + d_4x_4y_4 + d_5x_5y_5 + d_6x_6y_6 &= 0,
\end{align*}
with $\xv' = (x_1,x_2,x_3,0,0,0)$.

We use Lemma \ref{lem-homog} to set $x_1=x_2=0$ and $y_1=y_2=0$, which reduces the problem
to a diagonal one of the form
\begin{align*}
x_3y_3 + x_4y_4 + x_5y_5 + x_6y_6 &= 0,\\
l_{33}x_3y_3 + d_4x_4y_4 + d_5x_5y_5 + d_6x_6y_6 &= 0.
\end{align*}
By Lemma \ref{lem-diag} we can deal with this situation, if 
the coefficients $l_{33},d_4,d_5,d_6$ take on three different values.
Otherwise, we can assume that $d_5=d_6$.

Taking linear combinations in the original system, we therefore can simplify
our problem to 
\begin{align*}
x_1y_1 + x_2y_2 + x_3y_3 + x_4y_4 + x_5y_5 + x_6y_6 &= 0,\\
L_1(\xv')y_1 + L_2(\xv')y_2 + L_3(\xv')y_3 + d_4x_4y_4 &= 0.
\end{align*}
Since the second equations doesn't depend on $y_5$ and $y_6$
Lemma \ref{lem-homog} (iii) reduces the problem further to
\begin{align*}
x_5y_5 + x_6y_6 &= 0,\\
L_1(\xv')y_1 + L_2(\xv')y_2 + L_3(\xv')y_3 + d_4x_4y_4 &= 0.
\end{align*}
Both equations are independent of each other. The first gives a bound of
$O(N^2\log N)$ by Lemma \ref{lem-divest} and the second is fine as well
by the same argument (with a coordinate change before),
as long as the corresponding matrix has rank at least two.
Otherwise we are in the exceptional case (2) of Theorem \ref{Main-thm}.

\section{Extension to $s > 6$}  \label{Sec-deduction}

What happens, when the number of variables is larger then six? Either every linear combination
of the two matrices has rank at most five, which brings us to the exceptional case (1) in 
Theorem \ref{Main-thm}, or we can find a change of coordinates, such that our system looks
like 
\begin{equation} \label{eq-sys1}
\begin{split} 
x_1y_1 + \ldots + x_6y_6 + d_7x_7y_7 + \ldots + d_sx_sy_s= 0,\\
L_1(\xv)y_1 + \ldots + L_6(\xv)y_6 + L_7(\xv)y_7 +\ldots + L_s(\xv)y_s= 0.
\end{split}
\end{equation}
with $d_i \in \{0,1\}$.
Setting $x_i=y_i=0$ for all $i \geq 7$ by Lemma \ref{lem-homog}, 
we can use the result for $s=6$ to see that we either get the general result or that we can
add a multiple of the first equation to ensure that $L_1,\ldots,L_6$ are multiples of
each other.

We can apply the same argument for any set of six variables for which $d_i \neq 0$.
This results in the following structure for some value $f \geq 6$.
\begin{align*}
x_1y_1 + \ldots + x_fy_f = 0,\\
L(\xv)y_1 + \ldots + L(\xv)y_f + L_{f+1}(\xv)y_{f+1} +\ldots + L_s(\xv)y_s= 0.
\end{align*}
By Lemma \ref{lem-homog} this can be reduced to 
\begin{align*}
x_1y_1 + \ldots + x_fy_f = 0,\\
L_{f+1}(\xv)y_{f+1} +\ldots + L_s(\xv)y_s= 0,
\end{align*}
and we are done, as long as these linear forms $L_j$ are not all multiples of each other.

In the complementary case, the rank of the second equation in \eqref{eq-sys1} is at most
two. If it is less than two, we are done. 
Otherwise we perform a suitable change of coordinates, swap the equations, and obtain the form
\begin{align*}
x_1y_1 + x_2y_2= 0,\\
\tilde{L}_1(\xv)y_1 + \ldots + \tilde{L}_s(\xv)y_s= 0.
\end{align*}
for some other linear forms $\tilde{L}_1,\ldots,\tilde{L}_s$.
Since the second equation must have rank at least four, we can find two linear forms $L_i$ and $L_j$
for $i > j > 2$, which are linearly independent. 
An application of Lemma \ref{lem-homog} gives the system
\begin{align*}
x_1y_1 + x_2y_2= 0,\\
L_i(\xv)y_i + L_j(\xv)y_j= 0
\end{align*}
and we are done.
This is the end of the proof for Theorem \ref{Main-thm}.

\appendix

\section{Proof of Lemma \ref{Rank1-structure}} \label{AppA}

Let $B$ be a matrix with off-rank one.
By permuting variables, if necessary, we can assume that $B$
has the form
\begin{align*} \label{eq-matrix-decomp}
B = \begin{pmatrix}
a & r & \mv^T\\
s & b & \nv^T\\
\vv & \wv & C
\end{pmatrix},
\end{align*}
where $r \neq 0$ and $\mv,\nv,\vv,\wv \in \Z^{s-2}$. The following lemma
is the first step to understand the structure of $B$.

\begin{lemma} \label{lem-C-diag}
For the above matrix we have
\begin{align*}
C = \wv\mv^T/r + D,
\end{align*} 
where $D$ is a diagonal matrix.
\end{lemma}

\begin{proof}
Consider the $2\times 2$ submatrix $\begin{pmatrix} r & m_j \\ w_i & c_{ij} \end{pmatrix}$
for some $i \neq j$.
Since the off-rank is one, this matrix has rank at most one. Since $r \neq 0$
it must be at least one. A short calculation shows that $c_{ij}=w_im_j/r$.
\end{proof}

Since the off-rank of $B$ is one, we also get that 
$\vv = \lambda \uv, \wv = \mu \uv$ and $\mv = \alpha \kv, \nv = \beta \kv$
for some $\uv,\kv \in \Z^{s-2}\backslash\{\nullv\}$ and $\lambda, \mu,\alpha,\beta \in \Q$.
We obtain
\begin{align*}
B = \begin{pmatrix}
a & r & \alpha \kv^T\\
s & b & \beta \kv^T\\
\lambda \uv & \mu \uv & C
\end{pmatrix},
\end{align*}

{\bf Case 1}: $\alpha = 0$.\\
This implies that C is diagonal.\\

{\bf Case 1.1}: There are $i \neq j$ such that $u_i \neq 0$ and $k_j \neq 0$\\
Consider the off-diagonal matrix 
$\begin{pmatrix}
s & \beta k_j \\
\lambda u_i & c_{ij}  
\end{pmatrix}$.
Since $c_{ij}=0$, we conclude that $\beta = 0$ or $\lambda = 0$.\\

{\bf Case 1.1.1}: $\lambda = 0$.\\
We are in case $(ii)$ of Lemma \ref{Rank1-structure} and done.\\

{\bf Case 1.1.2}: $\beta = 0$ and $\lambda \neq 0$.\\
If $\mu = 0$, we are in case (ii) again. Otherwise
we have to show that we can choose entries $x,y$ in 
\begin{align*}
\begin{pmatrix}
x & r & \nullv^T\\
s & y & \nullv^T\\
\lambda \uv & \mu \uv & O
\end{pmatrix},
\end{align*}
such that the resulting matrix has rank one.
Choose $x=r\lambda/\mu$ and $y=s \mu/\lambda$.
It follows that $B = D + \vv \otimes \wv$ for some $\vv,\wv \in \Q^s$
and a diagonal matrix $D$, which corresponds to case (i).\\

{\bf Case 1.2}: For all $i \neq j$ we have $u_i = 0$ or $k_j = 0$.

The condition implies that $\uv = \nullv, \kv = \nullv$ or that 
there is at most one index $i$ such that 
$u_i\neq 0$ and $k_i \neq 0$. \\

{\bf Case 1.2.1}: $\uv = \nullv$.\\
We are in case (ii) of Lemma \ref{Rank1-structure}.\\

{\bf Case 1.2.2}: $\kv = \nullv$.\\
This brings us back to the Cases 1.1.1 and 1.1.2.\\

{\bf Case 1.2.3}: $u_j=k_j=0$ for all $j \neq i$ for some fixed $i$.\\
This implies that only the $1$st, $2$nd and $i$th
row/column have non-zero non-diagonal entries, which brings us into case (iii).\\

{\bf Case 2}: $\mu = 0$.\\
This is completely analogous to Case 1.\\

{\bf Case 3}: $\alpha \neq 0$ and $\mu \neq 0$.\\
Now there is at least one non-diagonal entry $c_{ij} \neq 0$. Consider the matrix
$\begin{pmatrix}
s & \beta k_j \\
\lambda u_i & c_{ij}  
\end{pmatrix}$.
We know that $c_{ij} = \mu u_i \alpha k_j/r$. This implies that the matrix
can have rank one only if $s = \lambda\beta u_ik_j/c_{ij}=\frac{r\lambda\beta}{\mu \alpha}$.
If we consider $B$ modulo diagonal matrices, we see that we can choose
$x$ and $y$ such that
\begin{align*}
\begin{pmatrix}
x & r & \alpha \kv^T\\
r\lambda\beta/\mu \alpha & y & \beta \kv^T\\
\lambda \uv & \mu \uv & \alpha\mu\uv\kv^T/r
\end{pmatrix},
\end{align*}
has rank one by setting $x = r\lambda/\mu$ and $y = r\beta/\alpha$.
This gives us case (i) in Lemma \ref{Rank1-structure}.


\begin{thebibliography}{HD}


\bibitem{Birch}
B. J. Birch, 
\newblock{\em Forms in many variables,} 
\newblock{Proc. Roy. Soc. Ser. A 265 (1961), 245--263.}


\bibitem{Dietmann}
R. Dietmann,
\newblock{\em Systems of rational quadratic forms,}
\newblock{Arch. Math. (Basel) 82 (2004), no. 6, 507--516.}


\bibitem{Schindler}
D. Schindler,
\newblock{\em Bihomogeneous forms in many variables,}
\newblock{J. Th\'eorie Nombres Bordeaux, to appear.}


\bibitem{Schmidt}
W. M. Schmidt,
\newblock{\em Simultaneous rational zeros of quadratic forms.}
\newblock{Seminar on Number Theory, Paris 1980-81 (Paris, 1980/1981), pp. 281--307,}
\newblock{Progr. Math., 22, Birkhäuser, Boston, Mass., 1982.} 


 

\end{thebibliography}
\end{document}